\documentclass[10pt,english,a4paper]{amsart}

\usepackage[T1]{fontenc}
\usepackage{lmodern}
\usepackage{babel}
\usepackage{amsmath,amssymb,amsthm}
\usepackage[marginratio={1:1,1:1},totalwidth=400pt,totalheight=600pt]{geometry}
\usepackage{microtype}
\usepackage{hyperref}

\newtheorem{theorem}{Theorem}[section]
\newtheorem{lemma}[theorem]{Lemma}
\newtheorem{corollary}[theorem]{Corollary}

\theoremstyle{definition}
\newtheorem{definition}[theorem]{Definition}
\newtheorem{remark}[theorem]{Remark}

\newcommand{\R}{\mathbb{R}}
\newcommand{\C}{\mathbb{C}}

\DeclareMathOperator{\spec}{sp}
\DeclareMathOperator{\supp}{supp}
\DeclareMathOperator{\re}{Re}
\DeclareMathOperator{\im}{Im}

\hypersetup{hidelinks,
pdftitle={Discrete hermitian groups are FC-hypercentral},
pdfauthor={Tron Omland}}

\title{Discrete hermitian groups are FC-hypercentral}
\author[Omland]{Tron Omland}
\address{Norwegian National Security Authority (NSM) \and Department of Mathematics, University of Oslo, Norway}
\email{tron.omland@gmail.com}
\date{\today}

\begin{document}

\begin{abstract}
We show that a discrete group is hermitian if and only if it is FC-hypercentral, that is, if and only if it does not have any nontrivial ICC quotients.
In particular, a finitely generated group is hermitian if and only if it is virtually nilpotent.
The main step is an elementary, self-contained proof that every hermitian ICC group is trivial, based on a finite-set separation property of nontrivial ICC groups.
\end{abstract}

\maketitle

\section{Introduction}

In this note all groups are discrete.
Let $G$ be a group.
We write the convolution product of $f,h\in\ell^1(G)$ as $fh$.
The involution is $f^*(g)=\overline{f(g^{-1})}$ and the norm is $\lVert f\rVert=\sum_{g\in G}|f(g)|$.
The point mass at $g\in G$ is denoted by $\delta_g$.
Thus $\delta_g\delta_h=\delta_{gh}$ and $\delta_g^*=\delta_{g^{-1}}$.
The unit is $1=\delta_e$ and we write $\lambda$ for $\lambda1$ when $\lambda\in\C$.
If $f$ and $h$ have disjoint supports then $\lVert f+h\rVert=\lVert f\rVert+\lVert h\rVert$.
All spectra are taken in $\ell^1(G)$.

Recall that $\ell^1(G)$ is \emph{hermitian} if $\spec(a)\subseteq\R$ for every $a=a^*\in\ell^1(G)$.
It is \emph{symmetric} if $\spec(a^*a)\subseteq[0,\infty)$ for every $a\in\ell^1(G)$.
The two properties are equivalent by the Shirali-Ford theorem, see~\cite[Theorem~11.4.1]{Palmer}, but we will not need this.
A group $G$ is hermitian if $\ell^1(G)$ is hermitian.
A group is \emph{ICC} if every conjugacy class other than $\{e\}$ is infinite.
The purpose of this note is to prove the following theorem.

\begin{theorem}\label{thm:main}
Let $G$ be a discrete ICC group.
If $\ell^1(G)$ is symmetric or hermitian, then $G$ is trivial.
\end{theorem}

Both properties fail as soon as there is a selfadjoint $a\in\ell^1(G)$ such that $a-i$ is not invertible.
Then $i\in\spec(a)$ and $\ell^1(G)$ is not hermitian.
Moreover $a^*a+1=(a+i)(a-i)$ is not invertible.
Indeed a product of two commuting elements is invertible only if both factors are invertible.
So $-1\in\spec(a^*a)$ and $\ell^1(G)$ is not symmetric.
Hence it suffices to show that $\ell^1(G)$ contains such an element whenever $G$ is a nontrivial ICC group.
The proof of Theorem~\ref{thm:main} is short, elementary, and self-contained.
Its only group-theoretic ingredient is the classical lemma on coset covers \cite[Lemma~4.1]{Neumann}, whose proof we include; the remaining argument uses norm estimates and a geometric series.
In particular, we do not use the result of Samei and Wiersma that hermitian groups are amenable \cite{SameiWiersma}.

Together with permanence under quotients (see \cite[Theorem~12.5.18(a)]{Palmer}), Theorem~\ref{thm:main} shows that every hermitian group is FC-hypercentral.
The converse follows from a transfinite adaptation of Ludwig's argument in \cite{Ludwig}, giving the characterization in Corollary~\ref{cor:characterization}.
We recall that every virtually nilpotent group is FC-hypercentral, that every FC-hypercentral group has polynomial growth (meaning that all its finitely generated subgroups have polynomial growth), and that these three properties coincide for finitely generated groups.
In particular, it follows from our results that a finitely generated group is hermitian if and only if it is virtually nilpotent.

Let us explain the idea.
Let $a\in\ell^1(G)$ be selfadjoint and suppose that $a-i$ is invertible.
Put $R=(a-i)^{-1}$ and $C=\lVert R\rVert$.
In a $C^*$-algebra we would have $C\le1$ since the distance from $i$ to the real line is $1$.
In $\ell^1(G)$ there is no such bound.
We show that if $G$ is a nontrivial ICC group, then there is a selfadjoint $a'\in\ell^1(G)$ with $\lVert a'-a\rVert\le1/(4C)$ such that $a'-i$ is invertible and $\lVert(a'-i)^{-1}\rVert\ge\frac{49}{48}C$.
Repeating this we get selfadjoint elements $a_n$ such that the norms $C_n=\lVert(a_n-i)^{-1}\rVert$ grow geometrically.
The steps $a_{n+1}-a_n$ have norm at most $1/(4C_n)$.
So they are summable and the sequence $a_n$ converges to a selfadjoint element $a$.
Each $a_n-i$ is invertible but the inverses blow up.
This forces $a-i$ to be noninvertible.

How can a small perturbation increase $C$ by a fixed factor?
Let $b$ be selfadjoint with $\lVert b\rVert\le1$ and let $0<\varepsilon<1/C$.
Then $\lVert\varepsilon bR\rVert\le\varepsilon C<1$, so $1+\varepsilon bR$ is invertible and its inverse is the absolutely convergent geometric series $\sum_{k\ge0}(-\varepsilon bR)^k$.
Since $a+\varepsilon b-i=(1+\varepsilon bR)(a-i)$ we get
\[
(a+\varepsilon b-i)^{-1}=R(1+\varepsilon bR)^{-1}=R-\varepsilon RbR+\varepsilon^2RbRbR-\dotsb.
\]
Suppose for a moment that $RbR$ and $R$ have disjoint supports and that $\lVert RbR\rVert\ge C^2/2$.
Then the first two terms have norm at least $C+\varepsilon C^2/2$.
The other terms have norm at most $C\sum_{k\ge2}(\varepsilon C)^k$.
With $\varepsilon=1/(4C)$ the gain is $C/8$ and the loss is at most $C/12$.
So the norm grows by a fixed factor.

The main difficulty is that cancellation in the product $RbR$ can make its norm small.
The norm of a product in $\ell^1(G)$ can be much smaller than the product of the norms.
For example let $x\neq e$ with $x^2=e$ and put $f=\delta_e+i\delta_x$.
Then $f^2=2i\delta_x$ has norm $2$ while $\lVert f\rVert^2=4$.
The remedy is to put a new group element $z$ between the two factors.
In the example
\[
f\delta_zf=\delta_z+i\delta_{xz}+i\delta_{zx}-\delta_{xzx}
\]
has norm $4$ as soon as the four elements $z,xz,zx,xzx$ are different.
In general we want the products $uzv$ with $u,v\in\supp f$ to be pairwise different and to lie outside $\supp f$.
Pairwise distinctness prevents cancellation and gives $\lVert f\delta_zf\rVert=\lVert f\rVert^2$.
The second condition makes $f$ and $f\delta_zf$ have disjoint supports.
One may think of $z$ as a new free letter.
Such elements exist for every finite support precisely when $G$ is a nontrivial ICC group.
This is the only place where the ICC property is used.
To find $z$ we need a classical fact about covers of a group by cosets \cite[Lemma~4.1]{Neumann}.

The point mass $\delta_z$ need not be selfadjoint.
To rule out additional overlap for the natural selfadjoint choice $\frac12(\delta_z+\delta_{z^{-1}})$, one could impose $uzv\neq u'z^{-1}v'$ for all $u,v,u',v'\in\supp f$.
This is a stronger condition which we prefer to avoid.
Instead we write $\delta_z=\re\delta_z+i\im\delta_z$ with $\re\delta_z$ and $\im\delta_z$ selfadjoint.
By the triangle inequality one of $f(\re\delta_z)f$ and $f(\im\delta_z)f$ has norm at least $\lVert f\rVert^2/2$.
This is where the factor $1/2$ above comes from.

Finally $R$ usually has infinite support.
We approximate $R$ by finitely supported functions.
The error in this approximation does not depend on $b$.
This is enough to pass to all of $\ell^1(G)$.

Section~\ref{sec:cosets} contains the coset lemma and the choice of $z$.
Section~\ref{sec:separation} proves the separation estimate.
Section~\ref{sec:proof} contains the perturbation estimate and the proof of Theorem~\ref{thm:main}.
Section~\ref{sec:fch} gives the FC-hypercentral characterization.

\section{Coset covers and separating elements}\label{sec:cosets}

The following lemma is a classical fact about coset covers.
It is a weak form of \cite[Lemma~4.1]{Neumann}, which shows that one of the subgroups has index at most the number of cosets in the cover.
We include the short proof.

\begin{lemma}\label{lem:cosets}
Let $H_1,\dots,H_m$ be subgroups of a group $G$.
If $G$ is the union of finitely many left cosets of $H_1,\dots,H_m$, then some $H_j$ has finite index in $G$.
\end{lemma}

\begin{proof}
We may assume that $H_1,\dots,H_m$ are distinct.
We use induction on $m$ for all groups at the same time.
If $m=1$ then $G$ is a finite union of cosets of $H_1$.
So $H_1$ has finite index.

Let $m\ge2$ and suppose that every $H_j$ has infinite index.
Then $H_m$ has infinitely many left cosets.
Choose a left coset $gH_m$ that does not occur in the cover.
It is disjoint from the cosets of $H_m$ in the cover.
So $gH_m$ is covered by the cosets of $H_1,\dots,H_{m-1}$ in the cover.
Multiplying by $g^{-1}$ and intersecting with $H_m$ we see that $H_m$ is a finite union of sets $xH_j\cap H_m$ with $j<m$.
If $y\in xH_j\cap H_m$ then $xH_j=yH_j$ and $y\in H_m$.
Hence $xH_j\cap H_m=y(H_j\cap H_m)$.
Thus $H_m$ is covered by finitely many left cosets of the subgroups $H_j\cap H_m$ with $j<m$.
By the induction hypothesis applied to the group $H_m$ some $H_j\cap H_m$ has finite index in $H_m$.
Then $H_m$ is contained in a finite union of left cosets of $H_j$.
The same holds for every left coset of $H_m$.
We replace each coset of $H_m$ in the cover by such a finite union.
This gives a cover of $G$ by finitely many left cosets of $H_1,\dots,H_{m-1}$.
By the induction hypothesis one of these subgroups has finite index in $G$.
This is a contradiction.
\end{proof}

For a finite set $K\subseteq G$ we write $K^2=\{uv:u,v\in K\}$ and $K^3=\{uvw:u,v,w\in K\}$.
We call $K$ symmetric if $K^{-1}=K$.

\begin{definition}\label{def:separating}
Let $K\subseteq G$ be a finite symmetric set with $e\in K$.
An element $z\in G$ \emph{separates} $K$ if
\begin{itemize}
\item[(i)] the elements $uzv$ with $u,v\in K$ are pairwise different, and
\item[(ii)] $uzv\notin K$ for all $u,v\in K$.
\end{itemize}
\end{definition}

Condition (i) says exactly that $z$ is a switching element for $K^2$ in the sense of \cite[Definition~2.3]{FHTV}.
Condition (ii) says that $z\notin K^3$.
Indeed $uzv\in K$ means that $z\in u^{-1}Kv^{-1}$, and the union of these sets over $u,v\in K$ is $K^{-1}KK^{-1}=K^3$.
Since $K^3$ is symmetric, (ii) also gives $uz^{-1}v\notin K$ for all $u,v\in K$.

\begin{lemma}\label{lem:separating}
Let $G$ be a nontrivial ICC group.
Then every finite symmetric set $K\subseteq G$ with $e\in K$ has a separating element.
\end{lemma}

\begin{proof}
Since $G$ is nontrivial and ICC it is infinite.
Put $X=K^2$.
For $x,y\in X\setminus\{e\}$ let
\[
S_{x,y}=\{z\in G:zxz^{-1}=y\}.
\]
Let $C_G(x)$ be the centralizer of $x$.
If $z_0\in S_{x,y}$ then $S_{x,y}=z_0C_G(x)$.
So every nonempty $S_{x,y}$ is a left coset of $C_G(x)$.
The index of $C_G(x)$ is the number of conjugates of $x$.
It is infinite since $x\neq e$.
Each point of $K^3$ is a left coset of the trivial subgroup.
This subgroup also has infinite index since $G$ is infinite.
Thus the nonempty sets $S_{x,y}$ and the points of $K^3$ are finitely many left cosets of subgroups of infinite index.
By Lemma~\ref{lem:cosets} they do not cover $G$.
Let $z$ lie outside all of them.
Then $z\notin K^3$ and so (ii) holds.

Suppose that $uzv=u'zv'$, where $u,v,u',v'\in K$.
Put $x=vv'^{-1}$ and $y=u^{-1}u'$.
Both lie in $X$ since $K$ is symmetric.
Moreover $zxz^{-1}=y$.
If $x\neq e$ then $y\neq e$ and $z\in S_{x,y}$.
This is not the case.
So $x=e$ and then also $y=e$.
This means that $u=u'$ and $v=v'$.
So (i) holds.
\end{proof}

\begin{remark}\label{rem:icc}
The conclusion of Lemma~\ref{lem:separating} holds only in nontrivial ICC groups.
Suppose that some $x\neq e$ has only finitely many conjugates.
Let $K$ be a finite symmetric set that contains $e$ and all conjugates of $x$.
For every $z\in G$ the element $y=zxz^{-1}$ lies in $K$ and $ezx=yze$.
The pairs $(e,x)$ and $(y,e)$ are different since $x\neq e$.
So no $z$ satisfies (i).
If $G$ is trivial then no $z$ satisfies (ii) for $K=\{e\}$.
Lemma~\ref{lem:cosets} is what allows us to handle finitely many conjugacy classes at the same time.
\end{remark}

\section{Separation}\label{sec:separation}

For $z\in G$ we use the real and imaginary parts of $\delta_z$,
\[
\re\delta_z=\frac{\delta_z+\delta_{z^{-1}}}{2}
\qquad\text{and}\qquad
\im\delta_z=\frac{\delta_z-\delta_{z^{-1}}}{2i}.
\]
Both are selfadjoint and have norm at most $1$.
Moreover $\delta_z=\re\delta_z+i\im\delta_z$.
If $z^2=e$ then $\re\delta_z=\delta_z$ and $\im\delta_z=0$.

The next lemma is the heart of the proof.
It says that a separating element moves a function away from itself without cancellation.
At least half of the norm survives the passage to a selfadjoint element.
This is enough.
The estimate holds for all of $\ell^1(G)$ up to an arbitrarily small error.

\begin{lemma}\label{lem:separation}
Let $G$ be a nontrivial ICC group.
Let $f\in\ell^1(G)$ and let $t\ge0$ and $\eta>0$.
Then there are $z\in G$ and $b\in\{\re\delta_z,\im\delta_z\}$ such that
\[
\lVert f-tfbf\rVert\ge\lVert f\rVert+\frac t2\lVert f\rVert^2-\eta.
\]
If $f$ has finite support then this also holds with $\eta=0$.
\end{lemma}

\begin{proof}
Suppose first that $f$ has finite support.
Let $K$ be a finite symmetric set that contains $e$ and $\supp f$.
By Lemma~\ref{lem:separating} there is $z\in G$ that separates $K$.
We have
\[
f\delta_zf=\sum_{u,v\in K}f(u)f(v)\delta_{uzv}.
\]
The elements $uzv$ in this sum are pairwise different by (i).
So no two terms cancel and
\[
\lVert f\delta_zf\rVert=\sum_{u,v\in K}|f(u)||f(v)|=\lVert f\rVert^2.
\]
Since $f\delta_zf=f(\re\delta_z)f+i\,f(\im\delta_z)f$ the triangle inequality gives
\[
\lVert f\rVert^2\le\lVert f(\re\delta_z)f\rVert+\lVert f(\im\delta_z)f\rVert.
\]
Choose $b\in\{\re\delta_z,\im\delta_z\}$ with $\lVert fbf\rVert\ge\lVert f\rVert^2/2$.
The function $fbf$ is supported in $KzK\cup Kz^{-1}K$.
This set is disjoint from $K$ by (ii).
So $f$ and $fbf$ have disjoint supports and
\[
\lVert f-tfbf\rVert=\lVert f\rVert+t\lVert fbf\rVert\ge\lVert f\rVert+\frac t2\lVert f\rVert^2.
\]

Now let $f$ be arbitrary.
Let $F\subseteq G$ be finite and let $f_0$ be the restriction of $f$ to $F$.
Put $r=\sum_{g\notin F}|f(g)|$.
Then $\lVert f-f_0\rVert=r$ and $\lVert f_0\rVert=\lVert f\rVert-r$.
Since $\sum_{g\in G}|f(g)|<\infty$ we can choose $F$ so large that
\[
r\bigl(1+2t\lVert f\rVert\bigr)\le\frac\eta2.
\]
Let $b$ be any element with $\lVert b\rVert\le1$.
Then
\[
(f-tfbf)-(f_0-tf_0bf_0)=(f-f_0)-t(f-f_0)bf-tf_0b(f-f_0).
\]
So the norms of $f-tfbf$ and $f_0-tf_0bf_0$ differ by at most
\[
r\bigl(1+t\lVert f\rVert+t\lVert f_0\rVert\bigr)\le r\bigl(1+2t\lVert f\rVert\bigr)\le\frac\eta2.
\]
This bound does not depend on $b$.
Moreover
\[
\lVert f\rVert+\frac t2\lVert f\rVert^2-\lVert f_0\rVert-\frac t2\lVert f_0\rVert^2
=r\Bigl(1+\frac t2\bigl(\lVert f\rVert+\lVert f_0\rVert\bigr)\Bigr)\le\frac\eta2.
\]
Let $z$ and $b$ be given by the first part applied to $f_0$.
Then
\[
\lVert f-tfbf\rVert\ge\lVert f_0-tf_0bf_0\rVert-\frac\eta2\ge\lVert f_0\rVert+\frac t2\lVert f_0\rVert^2-\frac\eta2\ge\lVert f\rVert+\frac t2\lVert f\rVert^2-\eta.
\qedhere
\]
\end{proof}

\section{Perturbation of the inverse and the proof}\label{sec:proof}

We now apply Lemma~\ref{lem:separation} to the inverse of $a-i$.
The terms of order two and higher are controlled by a geometric series.

\begin{lemma}\label{lem:amplification}
Let $G$ be a nontrivial ICC group.
Let $a\in\ell^1(G)$ be selfadjoint and suppose that $a-i$ is invertible.
Put $C=\lVert(a-i)^{-1}\rVert$ and $\varepsilon=1/(4C)$.
Then there are $z\in G$ and $b\in\{\re\delta_z,\im\delta_z\}$ such that $a+\varepsilon b-i$ is invertible and
\[
\lVert(a+\varepsilon b-i)^{-1}\rVert\ge\frac{49}{48}C.
\]
\end{lemma}

\begin{proof}
Put $R=(a-i)^{-1}$ and let $b$ be any element with $\lVert b\rVert\le1$.
Since $R(a-i)=1$ we have
\[
a+\varepsilon b-i=(1+\varepsilon bR)(a-i).
\]
Here $\lVert\varepsilon bR\rVert\le\varepsilon C=1/4$.
So $1+\varepsilon bR$ is invertible and its inverse is the geometric series $\sum_{k\ge0}(-\varepsilon bR)^k$.
Hence $a+\varepsilon b-i$ is invertible and
\[
(a+\varepsilon b-i)^{-1}=R(1+\varepsilon bR)^{-1}=R-\varepsilon RbR+E,
\]
where $E=\sum_{k\ge2}R(-\varepsilon bR)^k$.
The $k$th term of $E$ has norm at most $C4^{-k}$.
So $\lVert E\rVert\le C/12$ and
\[
\lVert(a+\varepsilon b-i)^{-1}\rVert\ge\lVert R-\varepsilon RbR\rVert-\frac C{12}.
\]
This holds for every $b$ with $\lVert b\rVert\le1$.
Now apply Lemma~\ref{lem:separation} with $f=R$, $t=\varepsilon$ and $\eta=C/48$.
It gives $z$ and $b\in\{\re\delta_z,\im\delta_z\}$ with
\[
\lVert R-\varepsilon RbR\rVert\ge C+\frac{\varepsilon C^2}2-\frac C{48}=C+\frac C8-\frac C{48}.
\]
For this $b$ we get
\[
\lVert(a+\varepsilon b-i)^{-1}\rVert\ge C+\frac C8-\frac C{48}-\frac C{12}=\frac{49}{48}C.
\qedhere
\]
\end{proof}

\begin{remark}\label{rem:constants}
The constants are not important.
With $\varepsilon=s/C$ the gain is $sC/2$ and the loss is at most $Cs^2/(1-s)$.
Every $s$ with $0<s<1/3$ gives a factor larger than $1$ when $\eta$ is small.
\end{remark}

\begin{proof}[Proof of Theorem~\ref{thm:main}]
Let $G$ be a discrete ICC group and suppose that $G$ is nontrivial.
As explained after Theorem~\ref{thm:main} it suffices to find a selfadjoint $a\in\ell^1(G)$ such that $a-i$ is not invertible.

Put $a_0=0$.
Then $(a_0-i)^{-1}=i$ has norm $C_0=1$.
Suppose that $a_n$ is selfadjoint and that $a_n-i$ is invertible.
Put $C_n=\lVert(a_n-i)^{-1}\rVert$ and $\varepsilon_n=1/(4C_n)$.
Lemma~\ref{lem:amplification} gives $z_n\in G$ and $b_n\in\{\re\delta_{z_n},\im\delta_{z_n}\}$ such that
\[
a_{n+1}=a_n+\varepsilon_nb_n
\]
is selfadjoint, $a_{n+1}-i$ is invertible and $C_{n+1}\ge\frac{49}{48}C_n$.
By induction $C_n\ge(49/48)^n$ for all $n$.
Hence
\[
\sum_{n=0}^\infty\lVert a_{n+1}-a_n\rVert\le\sum_{n=0}^\infty\varepsilon_n\le\frac14\sum_{n=0}^\infty\Bigl(\frac{48}{49}\Bigr)^n<\infty.
\]
So the sequence $a_n$ converges to
\[
a=\sum_{n=0}^\infty\varepsilon_nb_n.
\]
The involution is isometric and each $\varepsilon_nb_n$ is selfadjoint.
So $a$ is selfadjoint.

It remains to show that $a-i$ is not invertible.
Put $R_n=(a_n-i)^{-1}$ and $\xi_n=R_n/C_n$.
Then $\lVert\xi_n\rVert=1$ and
\[
(a-i)\xi_n=\frac{(a_n-i)R_n+(a-a_n)R_n}{C_n}=\frac{1+(a-a_n)R_n}{C_n}.
\]
Hence
\[
\lVert(a-i)\xi_n\rVert\le\frac1{C_n}+\lVert a-a_n\rVert.
\]
Both terms tend to zero.
Thus $(\xi_n)$ is a sequence of approximate eigenvectors for left multiplication by $a$, at $i$.
If $a-i$ had an inverse $S$ then
\[
1=\lVert\xi_n\rVert=\lVert S(a-i)\xi_n\rVert\le\lVert S\rVert\lVert(a-i)\xi_n\rVert
\]
for all $n$.
This is impossible since the right hand side tends to zero.
Hence $a-i$ is not invertible.
\end{proof}

\section{FC-hypercentral groups}\label{sec:fch}

We conclude by relating Theorem~\ref{thm:main} to the upper FC-central series.
The converse implication is outlined below using the argument in \cite{Ludwig}; the sketch is intended to be read alongside that paper.
Together with permanence under quotients, this gives the characterization in Corollary~\ref{cor:characterization}.

The \emph{FC-center} $\mathrm{FC}(G)$ of a group $G$ consists of the elements with finite conjugacy class.
The upper FC-central series is defined by
\[
F_0=\{e\},\qquad F_{\alpha+1}/F_\alpha=\mathrm{FC}(G/F_\alpha),\qquad
F_\lambda=\bigcup_{\alpha<\lambda}F_\alpha
\]
for every limit ordinal $\lambda$.
This series eventually stabilizes at a normal subgroup $F_\infty$, called the \emph{FC-hypercenter} of $G$.
The group $G$ is \emph{FC-hypercentral} if $F_\infty=G$.

\begin{theorem}\label{thm:fch-hermitian}
	Every FC-hypercentral group is hermitian.
\end{theorem}

\begin{proof}[Proof sketch]
	We use the notation of \cite[pp.~188--190]{Ludwig}.
	For a proper closed left ideal $I$ of $\ell^1(G)$,
	let $S_I^H$ denote the bounded positive sesquilinear forms
	that vanish on $I\times\ell^1(G)$ and are invariant under
	simultaneous left translation by $H$.
	Assume that $S_I=S_I^{\{e\}}\neq\{0\}$.
	By \cite[p.~188, (1)]{Ludwig}, it suffices to show that
	$S_I^G\neq\{0\}$.
	
	Let $(F_\alpha)_{\alpha\le\gamma}$ be the upper FC-central
	series, with $F_\gamma=G$.
	We show by transfinite induction that
	$S_I^{F_\alpha}\neq\{0\}$.
	The initial step is immediate.
	The successor step follows from \cite[Lemma~1]{Ludwig},
	since $F_{\alpha+1}/F_\alpha=\mathrm{FC}(G/F_\alpha)$
	is an s.i.\ subgroup in the terminology of \cite{Ludwig}.
	
	Let $\beta$ be a limit ordinal and suppose the assertion
	holds for all $\alpha<\beta$.
	Consider
	\[
	A_\alpha=\{B\in S_I^{F_\alpha}:
	\lVert B\rVert\le1,\ B(1,1)\ge\tfrac12\}.
	\]
	These sets are nonempty.
	Indeed, for a positive form,
	\[
	\lVert B\rVert=\sup_{x\in G}B(\delta_x,\delta_x),
	\]
	so translating and rescaling a nonzero form gives the
	required normalization.
	Simultaneous left translation preserves the annihilation condition because
	$I$ is a left ideal, and preserves $F_\alpha$-invariance
	because $F_\alpha$ is normal.
	
	The sets $A_\alpha$ are weak$^*$-compact and decrease with
	$\alpha$, so their intersection is nonempty.
	Every form in the intersection is nonzero and invariant under
	$F_\beta=\bigcup_{\alpha<\beta}F_\alpha$.
	This proves the limit step.
	Transfinite induction now gives
	$S_I^G=S_I^{F_\gamma}\neq\{0\}$, as required.
\end{proof}

We can now combine the two implications.

\begin{corollary}\label{cor:characterization}
	A discrete group is hermitian if and only if it is FC-hypercentral.
\end{corollary}

\begin{proof}
Suppose that $G$ is hermitian, and let $F_\infty$ be its FC-hypercenter.
Since the upper FC-central series has stabilized, $\mathrm{FC}(G/F_\infty)=\{e\}$, so $G/F_\infty$ is ICC.
Quotients of hermitian groups are hermitian \cite[Theorem~12.5.18(a)]{Palmer}.
Thus Theorem~\ref{thm:main} implies that $G/F_\infty$ is trivial, and hence $G$ is FC-hypercentral.
The converse is Theorem~\ref{thm:fch-hermitian}.
\end{proof}

\begin{remark}
The assertion in \cite[Theorem~12.5.17]{Palmer} that every locally nilpotent group is hermitian is incompatible with Theorem~\ref{thm:main}.
Indeed, there exist locally nilpotent ICC groups.
\end{remark}


\section*{AI disclosure and scope}
Around ten years ago, I worked on the problem of showing that non-trivial ICC groups cannot be hermitian.
The present result was obtained by feeding my notes and partial drafts from that period into AI models.
While many of the underlying ideas originated in those earlier drafts, AI assistance was essential for completing the proof.
For this reason, I have chosen to make this note available now in a concise preliminary form.
A later version will include a self-contained proof of Theorem~\ref{thm:fch-hermitian} and show how the results generalize to the twisted setting.
\vspace{2\baselineskip}

\end{document}